\documentclass[a4paper, 11pt]{amsart}

\usepackage{graphicx}
\usepackage[lmargin=1in,rmargin=1in,tmargin=1.25in,bmargin=1in]{geometry}
\usepackage{mathtools}
\usepackage[hidelinks]{hyperref}

\newtheorem{theorem}{Theorem} [section]
\newtheorem{lemma}[theorem]{Lemma}
\newtheorem{proposition}[theorem]{Proposition}
\newtheorem{corollary}[theorem]{Corollary}
\newtheorem{conjecture}[theorem]{Conjecture}

\theoremstyle{remark}
\newtheorem{remark}[theorem]{Remark}

\DeclareMathOperator\coker{coker}

\newcommand{\bz}{\mathbb Z}
\newcommand{\F}{\mathbb F}
\newcommand{\cT}{\mathcal T}

\newcommand{\con}{\mathcal C}
\newcommand{\cont}{\mathcal{C}^{TOP}}
\newcommand{\conts}{\mathcal {T}}

\begin{document}

\title[On independence of iterated doubles in the concordance group]{On independence of iterated Whitehead doubles in the knot concordance group}

\author{Kyungbae Park}

\address{School of Mathematics\\Korea Institute for Advanced Study\\Seoul 02455, Republic of Korea}
\email{kbpark@kias.re.kr}
\urladdr{\url{http://newton.kias.re.kr/~kbpark/}}

\begin{abstract}
Let $D(K)$ be the positively-clasped untwisted Whitehead double of a knot $K$, and $T_{p,q}$ be the $(p,q)$ torus knot. We show that $D(T_{2,2m+1})$ and $D^2(T_{2,2m+1})$ are linearly independent in the smooth knot concordance group $\con$ for each $m\geq 2$. Further, $D(T_{2,5})$ and $D^2(T_{2,5})$ generate a $\bz\oplus\bz$ summand in the subgroup of $\con$ generated by topologically slice knots. We use the concordance invariant $\delta$ of Manolescu and Owens, using Heegaard Floer correction term. Interestingly, these results are not easily shown using other concordance invariants such as the $\tau$-invariant of knot Floer theory and the $s$-invariant of Khovanov homology. We also determine the infinity version of the knot Floer complex of $D(T_{2,2m+1})$ for any $m\geq 1$ generalizing a result for $T_{2,3}$ of Hedden, Kim and Livingston.
\end{abstract}

\maketitle

\section{Introduction}
A knot $K$ in $S^3$ is called \emph{smoothly} (resp. \emph{topologically}) \emph{slice} if it bounds a smoothly (resp. locally flatly) embedded disk in $B^4$. Two knots $K_1$ and $K_2$ are called \emph{smoothly} (resp. \emph{topologically}) \emph{concordant} if $K_1\# -K_2$ is smoothly (resp. topologically) slice, where $-K$ is the mirror of $K$ with reversed orientation. Modulo smooth (resp. topological) concordance, the set of knots forms an abelian group, called the \emph{smooth} (resp. \emph{topological}) \emph{knot concordance group}, denoted by $\con$ (resp. $\con^{TOP}$).

Note that every smoothly slice knot is topologically slice, and hence we can consider the natural map from $\con$ to $\cont$. We denote the kernel of the map by $\conts$. Observe that non-trivial elements in $\conts$ are represented by knots which are topologically slice but not smoothly: hence $\conts$ portrays the remarkable distinction between the smooth and topological categories in dimension four. In particular a knot which is topologically slice but not smoothly, can be used to construct an exotic $\mathbb{R}^4$ (a manifold which is homeomorphic but not diffeomorphic to $\mathbb{R}^4$) \cite[Exercise 9.4.23]{Gompf-Stipsicz:1999}.

One rich source of non-trivial elements in $\conts$ has been the satellite operation called (positively-clasped untwisted) \emph{Whitehead doubling}, which is defined by the pattern in Fig.\@ \ref{fig:Whiteheaddouble}. By a result of Freedman \cite{Freedman:1982-1} the Whitehead double $D(K)$ of any knot $K$ is topologically slice, and hence its class is contained in $\conts$. It is thus important to understand the concordance properties of Whitehead double knots and also the effect of the operation on $\mathcal{C}$, and there is a long-standing conjecture.

\begin{figure}[t!]
	\centering
	\includegraphics[width=\textwidth]{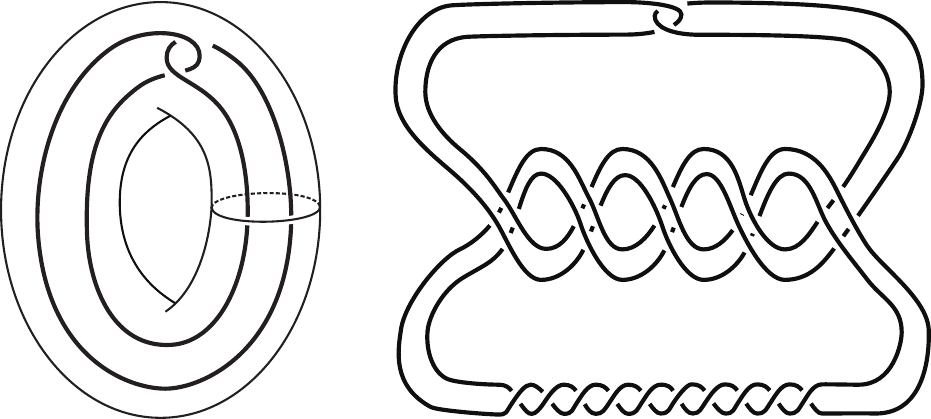}
	\caption{The pattern of the positively-clasped untwisted Whitehead double and the Whitehead double of the $(2,5)$ torus knot. The $-5$ extra full twists arise from untwisting the writhe of the projection of the $(2,5)$ torus knot.}
	\label{fig:Whiteheaddouble}
\end{figure}

\begin{conjecture}[{\cite[Problem 1.38]{Kirby:1997}}] $D(K)$ is smoothly slice if and only if $K$ is smoothly slice.\label{conj:Kirby}\end{conjecture}

Note that it is still unknown, as far as the author knows, if the conjecture is true even for some simple knots such as the left-hand trefoil or the figure-eight knot, i.e. if Whitehead doubles of them are smoothly slice or not. This shows that it is sometimes challenging to prove non-sliceness of the Whitehead double of a knot (or more generally of a topologically slice knot). Many classical concordance invariants, which is applied in the topological locally flat category, such as the knot signature, Levine's homomorphism \cite{Levine:1969}, and Casson-Gordon invariants \cite{Casson-Gordon:1975}, are not effective to study the sliceness of Whitehead double knots since these invariants vanish for those knots. It was not until early 1980's that the first example of topologically but not smoothly slice knot was known, due to the unpublished result of Akbulut and Casson. They used an obstruction from Donaldson's gauge theory to show that certain topologically slice knot, such as the Whitehead double of the right-handed trefoil knot, is not smoothly slice. 

Although many gauge theoretic techniques have allowed us to show non-sliceness of Whitehead doubles, one can ask further if there are pairs of Whitehead doubles which are linearly independent in $\conts$, or more strongly, if Whitehead doubles can generate a higher-rank summand in $\conts$. In \cite{Hedden-Kirk:2012-1}, Hedden and Kirk made use of Chern-Simon gauge theory to show Whitehead double knots can generate $\bz^\infty$-subgroup in $\conts$. Another direction to study these questions using recently developed knot concordance invariants from Heegaard Floer theory. The theory has provided manifestly \emph{smooth} concordance invariants, some of which give homomorphisms from $\mathcal{C}$ to $\mathbb{Z}$. One is the $\tau$-invariant, defined using the knot Floer homology of Ozsv\'{a}th-Szab\'{o} \cite{Ozsvath-Szabo:2003-3,  Ozsvath-Szabo:2004-1} and Rasmussen \cite{Rasmussen:2003-1}. Manolescu-Owens discovered another concordance invariant $\delta$, twice the Heegaard Floer correction term ($d$-invariant) of the double cover of $S^3$ branched over a knot \cite{Manolescu-Owens:2007-1}. More recently, Peters studied the concordance invariant $d(S^3_1(K))$ given by the correction term of $1$-surgery on $K\subset S^3$ \cite{Peters:2010-1}. In contrast to the other two invariants, $dS^3_1$ does not induce a homomorphism to $\mathbb{Z}$. We refer to a survey paper \cite{Hom:2017} of Hom about more recent development of Heegaard Floer theory on the knot concordance. Rasmussen's $s$-invariant coming from Khovanov homology is also a powerful smooth concordance homomorphism \cite{Rasmussen:2010-1}.

Even though there are many smooth concordance invariants, most of them are still ineffective for distinguishing Whitehead doubles in $\mathcal{C}$. The invariants $|\tau|$, $-dS^3_1/2$ and $|s/2|$ are known to be bounded above by the slice genus (four-ball genus) of the knot: the minimal genus of smoothly embedded surfaces in $B^4$ bounded by $K\subset\partial B^4$. Since the slice genus of $D(K)$ is at most one for any knot $K$, so are $|\tau|$, $-dS^3_1/2$ and $|s/2|$ of $D(K)$. Moreover, $\tau(D(K))$ is determined by $\tau(K)$ followed by a result of Hedden below.

\begin{theorem}[{\cite[Theorem 1.5]{Hedden:2007-1}}]
	\begin{equation*}
	\tau(D(K)) = \begin{cases} 0, & \mbox{for}\hspace{0.3cm} \tau(K)\leq 0\\1, & \mbox{for}\hspace{0.3cm} \tau(K)> 0.\end{cases}
	\end{equation*}
	\label{thm:Hedden}
\end{theorem}

In particular, $\tau(D^n(K))$ is identically either 0 or 1 for any $n\geq 1$ and is determined by $\tau(K)$, where $D^n(K)$ denotes the $n$th iterated positively clapsed untwisted Whitehead double of $K$. For the case of $s$-invariant, it is known by Livingston and Naik \cite[Theorem 2]{Livingston-Naik:2006-1} that $s(D(K))$ is sometimes determined by the Thurston-Bennequin number $TB$ of $K$ or $-K$. More precisely,
\begin{equation*}
-s(D(K))/2 = \begin{cases} 0 & \mbox{if}\hspace{0.3cm} TB(-K)\geq 0\\1 & \mbox{if}\hspace{0.3cm} TB(K)\geq 0.\end{cases}
\end{equation*}
Note that $TB(K)+TB(-K)\leq-1$. However, as far as the author knows, there is no known technique to determine the $s$-invariants of iterated Whitehead doubles in general. Therefore, it is interesting to ask if it is possible to distinguish the $D^n(K)$'s in $\con$. Using $\delta$-invariants, which are not constrained by the slice genus, we show:

\begin{theorem}\label{thm:general}
	Suppose $K$ is a knot in $S^3$. If $|\delta(D(K))|>8$, then $D(K)$ and $D^n(K)$ are not concordant for each $n\geq 2$. If $|\delta(D(K))|>8$ and $\tau(K)>0$, then $D(K)$ and $D^n(K)$ are linearly independent in $\con$ for each $n\geq 2$. 
\end{theorem}

For alternating knots, $\delta(D(K))$ is determined by the $\tau$ of $K$ by \cite[Theorem 1.5.]{Manolescu-Owens:2007-1}, and hence we have the following immediate corollary.
\begin{corollary}\label{cor:alternating}
	Suppose $K$ is an alternating knot in $S^3$. If $\tau(K)>2$, then $D(K)$ and $D^n(K)$ are linearly independent in $\con$ for each $n\geq 2$.
\end{corollary}

Let $T_{p,q}$ denote the $(p,q)$ torus knot. It is well known that $\tau(T_{2,2m+1})=m$. See the end of Section \ref{sec:staircase}. Therefore when $m>2$, it follows from the corollary above that $D(T_{2,2m+1})$ and $D^n(T_{2,2m+1})$ are linearly independent in $\con$ for each $n\geq 2$. In the case of $T_{2,5}$, for which $|\delta(D(T_{2,5}))|=8$, we study further to compute $\delta(D^2(T_{2,5}))$ and answer a question about the existence of a summand in $\conts$. 

\begin{theorem}\label{thm:torus}
	For each $m\geq 2$, $D(T_{2,2m+1})$ and $D^2(T_{2,2m+1})$ are linearly independent in $\con$. Furthermore, there exists a  $\mathbb{Z}\oplus\mathbb{Z}$-summand of $\conts$ generated by $D(T_{2,5})$ and $D^2(T_{2,5})$.
\end{theorem}

A tool for this result is a computation of the full knot Floer chain complex, $CFK^\infty$, of $D(T_{2,5})$. Here we study $CFK^\infty$  for $D(T_{2,2m+1})$, in general.
\begin{theorem}\label{thm:complex}
	For any $m\geq 1$, the chain complex $CFK^\infty(D(T_{2,2m+1}))$ is filtered chain homotopy equivalent to the chain complex $CFK^\infty(T_{2,3})\oplus A$, where $A$ is an acyclic complex. 
\end{theorem}

Recently, Cochran-Harvey-Horn suggested a bipolar filtration of $\con$ and the induced filtration of $\conts$ \cite{Cochran-Harvey-Horn:2013-1}, 
\[
	\{0\}\subset\dots\subset\cT_{n+1}\subset\cT_n\subset\dots\subset\cT_0\subset\cT=\conts.
\] 
Since $\tau$ of $D(K)$ and $D^2(K)$ are nonzero for knots $K$ in Theorem \ref{thm:torus}, both of them represent non-trivial classes in the same filtration level $\cT/ \cT_0$ by \cite[Corollary 4.9]{Cochran-Harvey-Horn:2013-1}. Using our knots, we get the following corollary related to the filtration. Let $\mathcal{C}_{\Delta}$ be the subgroup of $\conts$ generated by knots with trivial Alexander polynomial.  

\begin{corollary}
	There is a $\bz\oplus\bz$ summand of $\mathcal{C}_{\Delta}/(\mathcal{C}_{\Delta}\cap\cT_0)$. \label{cor:bipolar}
\end{corollary}

\begin{remark}
	In \cite[Corollary 1.9]{Manolescu-Owens:2007-1}, Manolescu and Owens showed the existence of a $\bz^2$-summand in $\mathcal{T}$ (actually in $\mathcal{C}_\Delta$) using $\tau$ and $\delta$ invariants for knots, $D(T_{2,3})$ and $D(T_{2,5})$. In \cite{Livingston:2008-1}, Livingston additionally used $s$ invariant to have a $\bz^3$-summand in $\mathcal{C}_\Delta$. More recently, Hom \cite{Hom:2015-1}, and Ozsv\'ath, Stipsicz and Szab\'o \cite{Ozsvath-Stipsicz-Szabo:2014} independently showed the existence of a $\bz^\infty$-summand in $\conts$ by using their $\epsilon$ and $\Upsilon$ invariant derived from knot Floer homology. See also a result of Kim and the author about the existence of a $\bz^\infty$-summand in $\con_\Delta$ \cite{Kim-Park:2016}. Also notice that Manolescu and Owens', and Kim and the author's results in fact show the existence of a $\bz^2$ and $\bz^\infty$-summand in $\mathcal{C}_\Delta/(\mathcal{C}_\Delta\cap\mathcal{T}_0)$ respectively. However, none of these techniques can be directly applied to see the independency of iterated Whitehead doubles of a single knot in $\con$. 
\end{remark}

\subsection*{Acknowledgments}
The author would like to thank his PhD supervisors, Matt Hedden for suggesting this problem as well as priceless guidance and Ron Fintushel for encouraging and supporting him. He also thanks Jae Choon Cha and Chuck Livingston for helpful communication and Tim Cochran for pointing out a wrong statement in an earlier version of this paper and suggesting Corollary \ref{cor:bipolar}. Finally, he is grateful to the anonymous referee for his/her valuable and helpful comments.

\section{Preliminaries}
\label{sec:preliminary}
In this section we briefly recall the Heegaard Floer and knot Floer theory, the staircase complexes and some invariants induced by them.

\subsection{Heegaard Floer homology and knot Floer invariants} 
For simplicity, we work with coefficients in  $\mathbb{F}$, the field of two elements. For a rational homology 3-sphere $Y$ equipped with a spin$^c$ structure $\mathfrak{t}$, one can associate to it a relatively $\mathbb{Z}$-graded and filtered chain complex $CF^\infty(Y,\mathfrak{t})$, a finitely and freely generated $\mathbb{F}[U,U^{-1}]$-module. In particular the filtration is obtained by the negative power of $U$, and $U$-multiplication lowers the homological grading by $2$. The filtered chain homotopy type of $CF^\infty(Y,\mathfrak{t})$ is known to be an invariant of $(Y,\mathfrak{t})$. For more detailed and general exposition of the theory we refer to \cite{Ozsvath-Szabo:2004-2, Ozsvath-Szabo:2004-3}.

We set ${CF}^-(Y,\mathfrak{t}):=CF^\infty(Y,\mathfrak{t})_{\{i<0\}}$, the subcomplex consisting of the elements in $CF^\infty(Y,\mathfrak{t})$ whose filtration level $i$ is less than $0$, and also define the quotient and sub-quotient complexes $CF^+(Y,\mathfrak{t}):=CF^\infty(Y,\mathfrak{t})_{\{i\geq0\}}$ and $\widehat{CF}(Y,\mathfrak{t}):=CF^\infty(Y,t)_{\{i=0\}}.$ The homology of $CF^\infty(Y,\mathfrak{t})$ is denoted by $HF^\infty(Y,\mathfrak{t})$, and $HF^-$, $HF^+$ and $\widehat{HF}$ denote the homology of the other chain complexes accordingly. The various versions of Heegaard Floer homologies naturally fit into a long exact sequence:
\begin{equation}
\dots\rightarrow HF^-(Y,\mathfrak{t})\rightarrow HF^\infty(Y,\mathfrak{t})\rightarrow HF^+(Y,\mathfrak{t})\rightarrow\dots
\label{eq:les}
\end{equation}
For $S^3$, it is known that $HF^\infty(S^3)\cong \mathbb{F}[U,U^{-1}]$ and $\widehat{HF}(S^3)\cong \mathbb{F}$ \cite{Ozsvath-Szabo:2004-3}. We usually drop the spin$^c$ structure in the notation if there is a unique one.

A knot $K$ in an integer homology three-sphere $Y$ has an associated $\mathbb{Z}\oplus\mathbb{Z}$ filtered chain complex $CFK^\infty(Y,K)$ which reduces to $CF^\infty(Y)$ after forgetting the second $\mathbb{Z}$ filtration. The $U$-multiplication decreases both of the filtration levels by $1$. The filtered chain homotopy type of $CFK^\infty(Y,K)$ is an invariant of the knot and we refer it as the knot Floer invariant of $(Y,K)$ \cite{Ozsvath-Szabo:2004-1, Rasmussen:2003-1}. We denote by $CFK^\infty(Y,K)_{\{(i,j)\}}$ the subgroup at $(i, j)$-filtration level in $CFK^\infty(Y,K)$ and define  $\widehat{CFK}(Y,K){:=}CFK^\infty(Y,K)_{\{i=0\}}$. It is an easy fact that $H_*(CFK^\infty(Y,K))\cong HF^\infty(Y)$ and $H_*(\widehat{CFK}(Y,K))\cong \widehat{HF}(Y)$. As a consequence, for a knot $K$ in $S^3$, we obtain an induced sequence of maps:
\[\iota^m_{K}\colon H_*(\widehat{CFK}(S^3,K)_{\{j\leq m\}})\rightarrow \widehat{HF}(S^3)\cong\mathbb{F}\]

An invariant $\tau$ for a knot $K$ in $S^3$ is defined by 
\[\tau(K){:=}\min\{m\in\mathbb{Z}{|}\iota^m_{K} \mbox{ is non-trivial}\}.\]
In \cite{Ozsvath-Szabo:2003-3, Rasmussen:2003-1} it was shown that $\tau$ is a concordance invariant. For a knot $K$ in $S^3$ we abbreviate the notations by $CFK^\infty(K):=CFK^\infty(S^3,K)$ and $\widehat{CFK}(K):=\widehat{CFK}(S^3,K)$.

It is useful to visualize a knot Floer complex as a collection of dots and arrows lying in a grid in the plane. In a diagram, a dot in $(i,j)$-coordinate box represents an $\mathbb{F}$-generator in $CFK^\infty(K)_{\{(i,j)\}}$, and an arrow represents the non-trivial map, $\mathbb{F}\rightarrow\mathbb{F}$. The differential is then the sum of the arrows, as a map of vector spaces. See Fig.\@ \ref{fig:cfk} for examples.

\subsection{Staircase complex}
\label{sec:staircase}
For a given $(n-1)$-tuple of positive integers $\mathbf{v}$, a \emph{staircase complex of length $n$}, St$(\mathbf{v})$, is defined as a finitely generated $\mathbb{Z}\oplus\mathbb{Z}$-filtered chain complex over $\mathbb{F}$ with $n$ generators, where the numbers in $\mathbf{v}$ are the length of arrows, which alternate horizontal and vertical starting at the top left generator and moving to the bottom right generator in alternating right and downward steps in a grid diagram. We also locate the top left dot on the vertical axis $(i=0)$ and the bottom right on the horizontal axis $(j=0)$ on the diagram. See \cite{Borodzik-Livingston:2014-1} for more detail. For instance, complexes generated by $\mbox{St}(1,1,1,1)$ and $\mbox{St}(1,2,2,1)$ are shown in Fig.\@ \ref{fig:cfk}. 

The knot Floer invariant is a categorification of Alexander polynomial $\Delta_K(t)$, in the following sense:
\[\Delta_K(t)=\sum_{k\in\mathbb{Z}}\chi(\widehat{CFK}(K)_{\{j=k\}}
)t^k,\]
where $\Delta_K(t)$ is the symmetrized Alexander polynomial of $K$ and $\chi$ is the Euler characteristic.

Conversely, if a knot is an $L$-space knot (including all torus knots), then its knot Floer homology is determined by its Alexander polynomial \cite{Ozsvath-Szabo:2005-1}. Suppose $K$ is an $L$ space knot, then its symmetrized Alexander polynomial has the form, 
\[\Delta_K(t)=\sum_{k=0}^{2m}(-1)^kt^{n_k},\]
where $n_k$ is a decreasing sequence of integers. The knot Floer homology of $K$ is generated by a stair complex, \[CFK^\infty(K)\cong\mbox{St}(n_{i+1}-n_{i})_{i=0}^{2m-1}\otimes\mathbb{F}[U,U^{-1}],\]
where $i$ runs from $0$ to $2m-1$ and $U$-multiplication is naturally extended: i.e. if $x$ is a generator in $(i,j)$-filtration level, then $U^nx$ sits in $(i-n,j-n)$-filtration level, and $\partial(U^nx)=U^n\partial(x)$. Denote $\mbox{St}(K):=\mbox{St}(n_{i+1}-n_i)_{i=0}^{2m-1}$.

For example, since the Alexander polynomial of $T_{2,2m+1}$ is $\sum_{i=-m}^{m}(-1)^it^i$, the full knot Floer chain complex of $T_{2,2m+1}$ is generated by $\mbox{St}(1,\dots,1)$ of length $2m+1$. The knot Floer complex of $T_{3,4}$ can be given by $\mbox{St}(1,2,2,1)\otimes\mathbb{F}[U,U^{-1}]$ from the fact that  $\Delta_{T_{3,4}}(t)=t^{-3}-t^{-2}+1-t^2+t^3$. See Fig.\@ \ref{fig:cfk}. Accordingly, it is easily obtained that $\tau(T_{p,q})=(p-1)(q-1)/2$ from its staircase complex, see also \cite[Corollary 1.7]{Ozsvath-Szabo:2003-3}. In particular, $\tau(T_{2,2m+1})$ is equal to $m$.

\begin{figure}[t!]
	\centering
	\includegraphics[width=\textwidth]{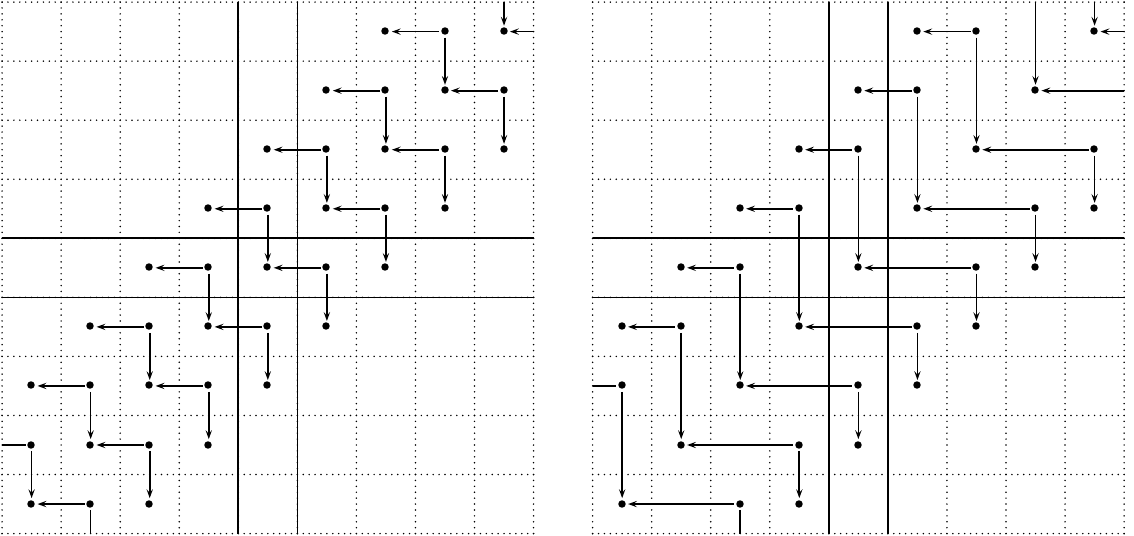}
	\caption{Diagrams of $CFK^\infty(T_{2,5})$ and $CFK^\infty(T_{3,4})$ generated by staircase complexes $\mbox{St}(1,1,1,1)$ and $\mbox{St}(1,2,2,1)$ respectively.}\label{fig:cfk}
\end{figure}

\subsection{Correction term and concordance invariants}
\label{sec:invariants}
If $Y$ is a rational homology 3-sphere, the homological grading in $HF^\circ(Y)$ can be lifted to the \emph{absolute $\mathbb{Q}$-grading} \cite{Ozsvath-Szabo:2006-1}. We usually write down the absolute grading using a subscript with parenthesis, for example, $\widehat{HF}(S^3)\cong\mathbb{F}_{(0)}$. With the help of the absolute grading, \emph{the correction term} (\emph{$d$-invariant}) of $(Y,\mathfrak{t})$ is defined \cite{Ozsvath-Szabo:2003-1}:
\[d(Y,\mathfrak{t}){:}{=}\min\{{gr}(\xi){|}\xi\in HF^\infty(Y,\mathfrak{t}) \mbox{ and } \pi(\xi) \mbox{ is nontrivial}\},\] where $\pi$ is the map from $HF^\infty(Y,\mathfrak{t})$ to $HF^+(Y,\mathfrak{t})$ in (\ref{eq:les}).

In \cite{Manolescu-Owens:2007-1}, Manolescu and Owens showed that the $d$-invariant of the double cover of $S^3$ branched over a knot $K$ is a concordance invariant for $K$:
\[\delta(K){:}=2d(\Sigma(K),\mathfrak{t}_0),\]
where $\Sigma(K)$ is the branched double cover of $S^3$ along $K$, and $\mathfrak{t}_0$ is the Spin$^c$ structure of $\Sigma(K)$ such that $c_1(\mathfrak{t}_0)=0\in H^2(\Sigma(K);\mathbb{Z})$. They also proved that $\delta$ of the Whitehead double of an alternating knot $K$ is determined by $\tau(K)$:
\begin{theorem}[{\cite[Theorem 1.5.]{Manolescu-Owens:2007-1}}]\label{thm:Manolescu_Owens}
	If $K$ is alternating, then \[\delta(D(K))=-4\max\{\tau(K),0\}.\]
\end{theorem}
Therefore, it easily follows from the theorem and the computation of $\tau$ in the previous section that $\delta(D(T_{2,2m+1})){=}-4m$. See also Section \ref{subsec:dinvariant}.

The $d$-invariant of the 1-surgery of $S^3$ along the knot $K$, $d(S^3_1(K))$, is also a concordance invariant studied by Peters \cite{Peters:2010-1}. There is an algorithm of computing $d(S^3_1(K))$, provided knowledge of $CFK^\infty(K)$. See also \cite[Section 3.2.2.]{Hom:2017}.

Let $C$ be $CFK^\infty(K)$. Consider the following inclusion map 
\[v^-_0\colon C\{\max(i,j)\leq0\}\rightarrow C\{i\leq0\},\]
then 
\begin{equation}\label{eq:d_1}
dS^3_1(K)=-2\dim_\F(\coker(v^-_0)_*).
\end{equation}

\section{The invariants $\delta$ of iterated Whitehead doubles}\label{sec:generalization}
In this section, we use a genus-bound property of the concordance invariant $dS^3_1$ to show that, provided that $|\delta(D(K))|>8$, $D(K)$ and $D^n(K)$ are not smoothly concordant for each $n\geq 2$. Secondly, we present formulas to compute $dS^3_1(K)$ and $\delta(D(K))$ for a given staircase complex of $K$ introduced in Section \ref{sec:staircase}.

\subsection{A bound on $\delta$ and proof of Theorem \ref{thm:general}}
First, we present a lemma that relates the $\delta$-invariant of a Whitehead double to $dS^3_1$.
\begin{lemma}
	\label{lemma:delta}
	For any knot $K$, $\delta(D(K))=2dS^3_1(K\#K^r)$.
\end{lemma}
\begin{proof}
	Let $S^3_p(K)$ denote the 3-manifold obtained by $p$-surgery of $S^3$ along a knot $K$. Manolescu-Owens showed that $d(S^3_{-1/2}(K))=d(S^3_{-1}(K))$ for any knot in the proof of \cite[Proposition 6.2]{Manolescu-Owens:2007-1}; in fact both of them equal $2h_0(K)$, where $h_0(K)$ is an invariant defined by Rasmussen in \cite[Section 7.2.]{Rasmussen:2003-1}. Thus, using the behaviour of $d$-invariants under orientation reversal \cite[Propostion 4.2]{Ozsvath-Szabo:2003-1}, we have
	\begin{equation*}
	\begin{split}
	d(S^3_{1/2}(K))&=-d(S^3_{-1/2}(-K))\\
	&=-d(S^3_{-1}(-K))\\
	&=d(S^3_1(K)).
	\end{split}
	\end{equation*}
	
	Recall that the double cover of $S^3$ branched over $D(K)$, $\Sigma(D(K))$, can be obtained by $1/2$-surgery along $K\#K^r$ in $S^3$, where $K^r$ is the knot $K$ with its orientation reversed, see \cite[Proposition 6.1]{Manolescu-Owens:2007-1}. From the definition of $\delta$-invariant,
	\begin{equation*}
	\begin{split}
	\delta(D(K))&:=2d(\Sigma(D(K)))\\
	&=2d(S^3_{1/2}(K\#K^r))\\
	&=2d(S^3_1(K\#K^r)).
	\end{split}
	\end{equation*} 
\end{proof}

This lemma, together with the genus bound property of $dS^3_1$, gives a bound on $\delta$-invariant of iterated Whitehead double knots.
\begin{proposition}
	$\delta(D^n(K))\geq-8$ for any knot $K$ and $n\geq2$.
\end{proposition}
\begin{proof}
	It is shown in \cite[Theorem 1.5.]{Peters:2010-1} that $-dS^3_1(K)/2$ is a lower bound for the slice genus, $g^*(K)$, of $K$, and note that the slice genus of $D(K)$ is at most 1 for any knot $K$. Hence, for $n\geq2$,
	\[-\delta(D^n(K))=-2dS^3_1(D^{n-1}(K)\#D^{n-1}(K)^r)\leq 4g^*(D^{n-1}(K)\#D^{n-1}(K)^r)\leq 8.\]
\end{proof}

\begin{proof}[Proof of Theorem \ref{thm:general}]
	Therefore, if $\delta(D(K))<-8$, or equivalently $dS^3_1(K\#K)<-4$, then $D(K)$ and $D^n(K)$ are not smoothly concordant. (According to \cite[Theorem 1.5]{Manolescu-Owens:2007-1}, $\delta(D(K))$ is nonpositive for any knot $K$.)
	
	Additionally if $\tau(K)>0$, then $\tau(D^i(K))=1$ for $i\geq1$ by Theorem \ref{thm:Hedden}. Considering the homomorphism $(\tau, \delta)$ from $\con$ to $\bz\oplus\bz$, we see that $D(K)$ and $D^n(K)$ are linearly independent in $\con$ for each $n\geq2$.
\end{proof}

\subsection{$dS^3_1(K)$ and $\delta(D(K))$ of a Staircase Complex}
\label{subsec:dinvariant}

If a knot admits a knot Floer complex generated by a staircase complex (equivalently, $L$-space knots), then its $dS^3_1$-invariant can be easily obtained. For the $d$-invariants of higher surgery coefficients, we refer to \cite[Section 4.2]{Borodzik-Livingston:2014-1}. 

\begin{proposition}
	Suppose the knot Floer complex of $K$ can be given by a staircase complex $\mbox{St}(K)$, then
	\[d(S^3_1(K))=-2\min_{(i,j)\in \mathrm{Vert}(\mathrm{St}(K))}\max\{i,j\}\]
	\[\delta(D(K))/2=d(S^3_1(K\#K))=-2\min_{(i,j),(k,l)\in \mathrm{Vert}(\mathrm{St}(K))}\max\{i+k,j+l\},\]
	where $\mathrm{Vert}(\mathrm{St}(K))$ is the set of the coordinates of the generators of $\mathrm{St}(K)$.
	\label{prop}
\end{proposition}

\begin{proof}
	Suppose that $CFK^\infty(K)$ is generated by $\mbox{St}(K)$, then the top left element in $\mbox{St}(K)$ represents the generator of $H_*({CFK}^\infty(K)_{\{i\leq0\}})\cong\mathbb{F}[U]$. The chain complex $\mbox{St}(K)$ has the form $0\rightarrow\mathbb{F}^k_{(+1)}\rightarrow \mathbb{F}_{(0)}^{k+1}\rightarrow 0$. For $\eta$ with $(i,j)$-coordinates, $U^{\max\{i,j\}+1}\eta$ lies in the subcomplex $CFK^\infty(K)_{\{i<0 \mbox{ and } j<0\}}$, whereas $U^{\max\{i,j\}}\eta$ does not. Note also that since $\mathrm{St}(K)$ is a $\bz\oplus\bz$-filtered complex, $\min_{(i,j)\in \mathrm{Vert}(\mathrm{St}(K))}\max\{i,j\}$ is realized by the elements in $\mathbb{F}_{(0)}^{k+1}$, not ones in $\mathbb{F}^k_{(+1)}$. Hence, the first formula follows from the Eq.\@ (\ref{eq:d_1}). 
	
	Although $CFK^\infty(K\#K)$ is not generated by a staircase complex, it can be constructed from the tensor complex, $\mbox{St}(K)\otimes \mbox{St}(K)$ by the connected-sum formula \cite[Theorem 7.1]{Ozsvath-Szabo:2004-1}. The coordinates of the generators in $\mbox{St}(K)\otimes \mbox{St}(K)$ are given by the sums of a pair of coordinates of the generators of $\mbox{St}(K)$. The complex $\mbox{St}(K)\otimes \mbox{St}(K)$ has the form $0\rightarrow \mathbb{F}^{k^2}_{(+2)}\rightarrow\mathbb{F}^{2k(k+1)}_{(+1)}\rightarrow\mathbb{F}^{(k+1)^2}_{(0)}\rightarrow 0$, and the generators with $(0)$-grading are homologous to the generator of $H_*({CFK}^\infty(K\#K)_{\{i=0\}})\cong\mathbb{F}$. See Fig.\ \ref{fig:tensor} for the tensor complex of two copies of $\mbox{St}(1,2,2,1)$. Therefore, we get the second formula similarly.
\end{proof}

\begin{figure}[h]
	\centering
	\includegraphics[width=0.5\textwidth]{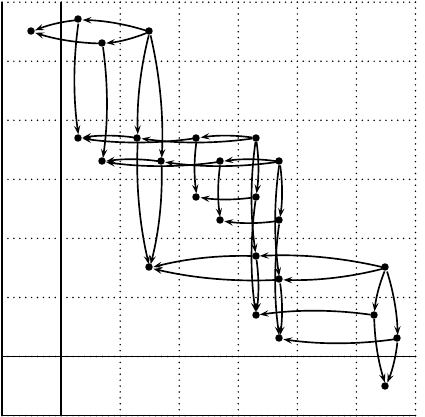}
	\caption{A diagram of the complex $\mbox{St}(T_{3,4})\otimes\mbox{St}(T_{3,4})$ generating $CFK^\infty(T_{3,4}\#T_{3,4})$}
	\label{fig:tensor}
\end{figure}

For example, since $\mbox{St}(T_{2,2m+1})=(1,\dots,1)$ of length $2m+1$, one can compute \[\delta(D(T_{2,2m+1}))=-4m.\] In the case of $(3,4)$ torus knot, $\mbox{St}(T_{3,4})=(1,2,2,1)$, and so \[\mbox{Vert}(T_{3,4})=\{(0,3),(1,3),(1,1),(3,1),(3,0)\}.\]
Thus $\delta(D(T_{3,4}))=-8$, and hence we cannot show that $D(T_{3,4})$ and $D^2(T_{3,4})$
are not concordant, using Theorem \ref{thm:general}. Note that there are many knots such that $\delta(D(K))=-8$: for example, any knot $K$ whose $CFK^\infty$ is generated by $\mbox{St}(1,n,n,1)$. 

\section{The full knot Floer chain complex of $D(T_{2,2m+1})$}
\label{sec:computation}
Recently, Hedden, Kim and Livingston showed that $CFK^\infty(D(T_{2,3}))$ is chain homotopy equivalent to $CFK^\infty(T_{2,3})\oplus A$ for some acyclic complex $A$, \cite[Proposition 6.1.]{Hedden-Kim-Livingston:2016-1}. See also \cite[section 9.1]{Cochran-Harvey-Horn:2013-1}. In this section we will generalize this result to $T_{2,2m+1}$ for $m\geq 1$, and furthermore the filtered chain homotpy type of $CFK^\infty(D(T_{2,2m+1}))$ will be completely determined.

Before proving Theorem \ref{thm:complex}, recall the following useful lemma regarding how a basis change in a filtered chain complex over $\mathbb{F}$ affects the diagram of a knot Floer chain complex.

\begin{lemma}[{\cite[Lemma A.1.]{Hedden-Kim-Livingston:2016-1}}]
	Let $C_*$ be a knot Floer complex with a 2-dimensional arrow diagram $D$ given by an $\mathbb{F}$-basis. Suppose that $x$ and $y$ are two basis elements of the same grading such that each of the $i$ and $j$ filtrations of $x$ is not greater than that of $y$. Then the $\mathbb{Z}\oplus\mathbb{Z}$ filtered basis change given by $y'=y+x$ gives rise to a diagram $D'$ of $C_*$ which differs from $D$ only at $y$ and $x$ as follows:
	\begin{itemize}
		\item Every arrow from some $z$ to $y$ in $D$ adds an arrow from $z$ to $x$ in $D'$
		\item Every arrow from $x$ to some $w$ in $D$ adds an arrow from $y'$ to $w$ in $D'$
	\end{itemize}
\end{lemma}

We use the above lemma for the purpose of removing certain boundary arrows in chain complexes over $\mathbb{F}$. For example, the proposition below will be useful for proving Theorem \ref{thm:complex}. 

\begin{proposition}
	Suppose $C$ is one of the $\mathbb{Z}\oplus\mathbb{Z}$ filtered chain complexes over $\mathbb{F}$ given by the diagrams depicted in Fig.\@ \ref{fig:uvw} with any possible combination of dotted arrows that make $C$ a chain complex. Then all dotted arrows in $C$ can be removed by a basis change.
	\label{cor2}
\end{proposition}

\begin{proof}
	First, consider the complex (I). Suppose that
	\begin{align*}
		\partial a&=b+c+Ax+By,\\ 
		\partial b&=d+Cz,\mbox{and}\\ 
		\partial c&=d+Dz
	\end{align*}
	for some $A$, $B$, $C$ and $D$ in $\mathbb{F}$. Since $\partial^2=0$,
	\begin{equation}\label{eq:1}
		\begin{split}
			0=\partial^2a&=\partial(b+c+Ax+By)\\
			&=(A+B+C+D)z.
		\end{split}
	\end{equation}
	Therefore, the coefficients have to satisfy the equation that $A+B+C+D=0$. Now, we consider every possible coefficient of $A$, $B$, $C$ and $D$ in $\mathbb{F}$ satisfying the equation and show that each case can be transformed to have $A=B=C=D=0$, as desired, after proper change of basis.
	
	\begin{itemize}
		\item If $A{=}B{=}1$ and $C{=}D{=}0$, change the basis by $b'{=}b+x$, $c'{=}c+y$ and $d'{=}d+z$.
		\item If $A{=}C{=}1$ and $B{=}D{=}0$, change the basis by $b'{=}b+x$.
		\item If $A{=}D{=}1$ and $B{=}C{=}0$, change the basis by $b'{=}b+x$ and $d'{=}d+z$.
		\item If $B{=}C{=}1$ and $A{=}D{=}0$, change the basis by $c'{=}c+y$ and $d'{=}d+z$.
		\item If $B{=}D{=}1$ and $A{=}C{=}0$, change the basis by $c'{=}c+y$.
		\item If $C{=}D{=}1$ and $A{=}B{=}0$, change the basis by $d'{=}d+z$.
		\item If $A{=}B{=}C{=}D{=}1$, change the basis by $b'{=}b+x$ and $c'{=}c+y$.
	\end{itemize}
	
	Similar argument is applied to remove any combination of possible dotted arrows in the complexes (II) and (III).
\end{proof}

\begin{figure}[htp!]
	\includegraphics[width=\textwidth]{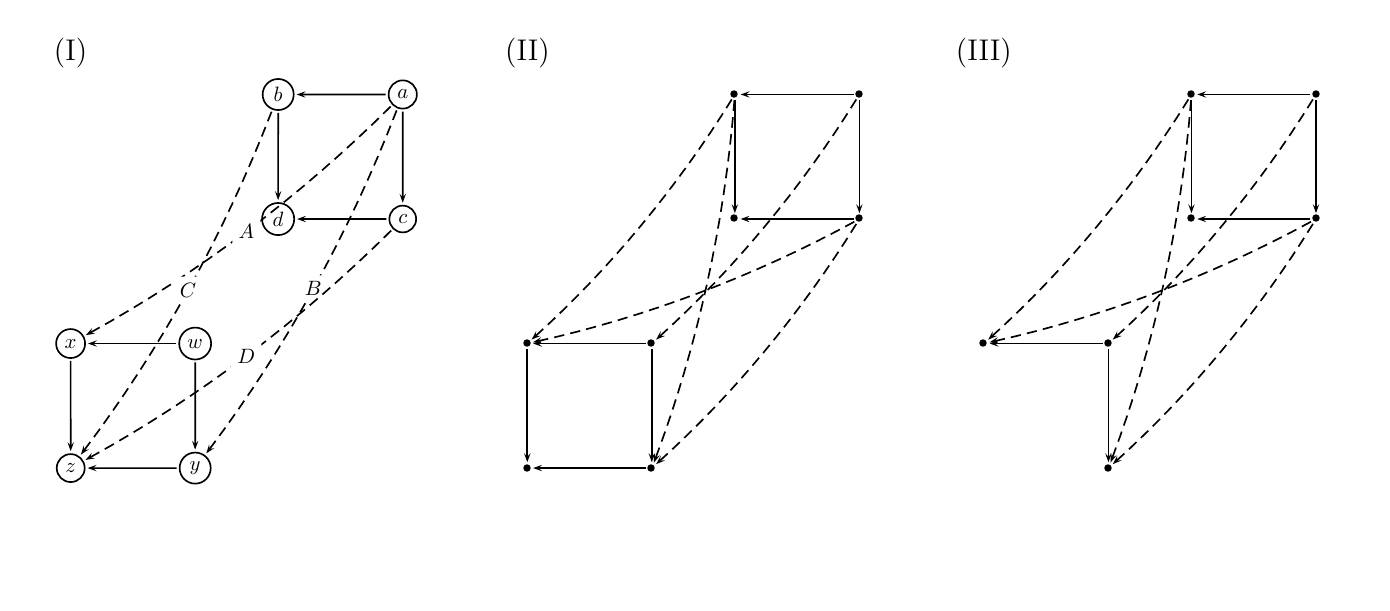}
	\caption{Any possible combination of dotted arrows can be removed by a basis change.}
	\label{fig:uvw}
\end{figure}

\begin{proof}[Proof of Theorem \ref{thm:complex}]
	Fix $m\geq 1$. Since $T_{2,2m+1}$ is an alternating knot, $\widehat{CFK}(T_{2,2m+1})$ is easily determined by the Alexander polynomial and signature of $T_{2,2m+1}$  as follows \cite{Ozsvath-Szabo:2003-2}:
	\begin{equation*}
	\widehat{CFK}(T_{2,2m+1}, j)=
	\begin{cases}
	\F\langle e_{m-j}\rangle&-m\leq j\leq m\\
	0& \text{otherwise}
	\end{cases}
	\end{equation*}
	$\F\langle e_{m-j}\rangle$ is supported in homological degree $(m-j)$, and $\partial e_{-2k+1}=e_{-2k}$ for $1\leq k\leq m$ and other boundary maps are trivial.
	
	In \cite[Theorem 1.2]{Hedden:2007-1} Hedden gave a formula to compute $\widehat{HFK}$ of Whitehead double of a knot $K$ provided $\widehat{CFK}$ of $K$. Let $D$ be $D(T_{2,2m+1})$ for $m\geq 1$. Applying Hedden's formula to $T_{2,2m+1}$, we have
	\begin{equation*}
	\widehat{HFK}_*(D,j)=
	\begin{cases}
	\mathbb{F}^{2m}_{(0)}\oplus\mathbb{F}^2_{(-1)}\oplus\mathbb{F}^2_{(-3)}\oplus\dots\oplus\mathbb{F}^2_{(-2m+1),}&j=1\\
	\mathbb{F}^{4m-1}_{(-1)}\oplus\mathbb{F}^4_{(-2)}\oplus\mathbb{F}^4_{(-4)}\oplus\dots\oplus\mathbb{F}^4_{(-2m),}&j=0\\
	\mathbb{F}^{2m}_{(-2)}\oplus\mathbb{F}^2_{(-3)}\oplus\mathbb{F}^2_{(-5)}\oplus\dots\oplus\mathbb{F}^2_{(-2m-1),}&j=-1\\
	0&otherwise.
	\end{cases}
	\end{equation*}
	We assign an $\mathbb{F}$-basis to each summand in the direct decomposition as below:
	\begin{multline*}
	\widehat{HFK}_*(D,j)\\=
	\begin{cases}
	\langle x^0_1,\dots,x^0_{2m}\rangle\oplus\langle u^{-1}_{1,1},u^{-1}_{1,2}\rangle\oplus\dots\oplus\langle u^{-2m+1}_{m,1},u^{-2m+1}_{m,2}\rangle&j=1\\
	\langle y^{-1}_1,\dots,y^{-1}_{4m-1}\rangle\oplus\langle v^{-2}_{1,1},\dots,v^{-2}_{1,4}\rangle\oplus\dots\oplus\langle v^{-2m}_{m,1},\dots,v^{-2m}_{m,4}\rangle&j=0\\
	\langle z^{-2}_1,\dots,z^{-2}_{2m}\rangle\oplus\langle w^{-3}_{1,1},w^{-3}_{1,2}\rangle\oplus\dots\oplus\langle w^{-2m-1}_{m,1},w^{-2m-1}_{m,2}\rangle&j=-1\\
	0&otherwise,
	\end{cases}\label{eq:generators}
	\end{multline*}
	where the superscript of a generator represents its absolute grading.
	
	Since $\widehat{HFK}_*(D)$ is homotopic equivalent to the $\widehat{CFK}(D)$ \cite[Lemma 4.5]{Rasmussen:2003-1}, we assume that \[CFK^\infty(D)_{\{(0,j)\}}=\widehat{HFK}_*(D,j)\] and \[CFK^\infty(D)_{\{(i,j)\}}\cong U^{-i}CFK^\infty(D)_{\{(0,j)\}}=\widehat{HFK}_*(D,j-i).\] Now, we investigate all differentials in $CFK^\infty(D)$ by using the following facts: $\partial^2=0$, $H_*(\widehat{CFK}(D))\cong\widehat{HF}(S^3)\cong \mathbb F_{(0)}$ and $H_*(CFK^\infty(D))\cong HF^\infty(S^3)\cong \mathbb{F}[U,U^{-1}]$.
	
	First, note that there are no components of the boundary maps between generators of the same $(i,j)$-filtration since they would reduce $\widehat{HFK}_*(D)$. Thus, we can decompose the boundary maps $\partial$ to the vertical, horizontal and diagonal components, $\partial=\partial_V+\partial_H+\partial_D$. Also, we remark that it is enough to determine boundary maps of $\mathbb{F}[U,U^{-1}]$-generators in $CFK^\infty$ because the boundary map is $U$-equivariant.
	
	Similar argument from \cite[Proposition 6.1]{Hedden-Kim-Livingston:2016-1} can be used to determine $\partial_V$ and $\partial_H$ i.e. by the fact that $H_*(\widehat{CFK}(D))=\mathbb{F}_{(0)}$ and using grading consideration, after changing basis, we can assume that
	\begin{align*}
		\partial_V(x^d_k)&=y^{d-1}_{k-1} &\mbox{for } k=2,\dots, 2m\\
		\partial_V(y^{d-1}_{2m+l-1})&=z^{d-2}_l &\mbox{for } l=1,\dots,2m.\\
		\partial_V(u^{d-1}_{p,i})&=v^{d-2}_{p,i}\mbox{ and }\partial_V(v^{d-2}_{p,i+2})=w^{d-3}_{p,i}&\mbox{for } p=1,\dots,m \mbox{ and } i=1,2,
	\end{align*}
	for $d\in 2\mathbb{Z}$ and $\partial_V$'s of other elements are trivial. Analogously, since $H_*(CFK^\infty(D)_{\{j=0\}})$ is isomorphic to $\widehat{HF}(S^3)\cong\mathbb{F}_{(0)}$ and $\partial^2=0$, the horizontal boundary components can be assumed as following:
	\begin{align*}\label{eq:v}
		\partial_H(z^d_k)&=y^{d-1}_{k-1} &\mbox{for } k=2,\dots, 2m\\
		\partial_H(y^{d-1}_{2m+l-1})&=x^{d-2}_l &\mbox{for } l=1,\dots,2m.\\
		\partial_H(w^{d-1}_{p,i})&=v^{d-2}_{p,i}\mbox{ and }\partial_H(v^{d-2}_{p,i+2})=u^{d-3}_{p,i}&\mbox{for } p=1,\dots,m \mbox{ and } i=1,2,
	\end{align*}
	and $\partial_H$'s of other elements are trivial. We drop $\mathbb{F}[U,U^{-1}]$ coefficients of generators since they are canonically determined by the grading considerations. See Fig.\@ \ref{fig:5} for a diagram. In fact we will show that we can assume there are no $\partial_D$ components for any elements as shown in Fig.\@ \ref{fig:5}.
	
	\begin{figure}[p!]
		\centering
		\includegraphics[width=0.98\textwidth]{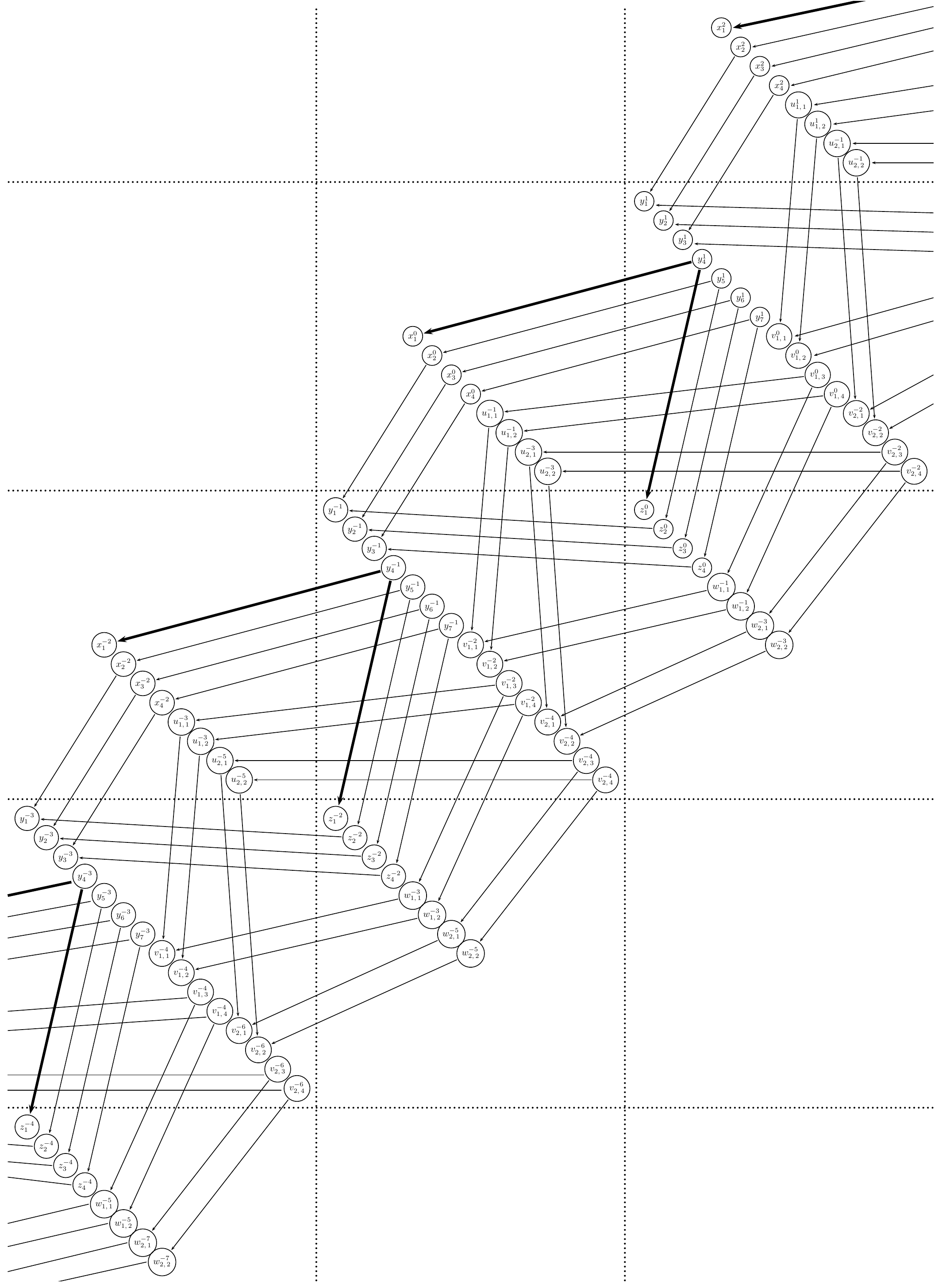}
		\caption{A diagram of $CFK^\infty(D(T_{2,5}))$. The superscript of each generator represents its grading. Note that the subcomplex generated by the bold arrows is a direct summand of the chain complex and isomorphic to $CFK^\infty(T_{2,3})$, and the homology of the complement of the summand is trivial.}
		\label{fig:5}
	\end{figure}
	
	We can split $CFK^\infty(D)$ as following disjoint subsets:
	\begin{align*}
		A^d_{p,i}&:=\{v^{d}_{p,i+2},u^{d-1}_{p,i},w^{d-1}_{p,i},v^{d-2}_{p,i}\},\\
		B^d_q&:=\{y^{d-1}_{2m+q},x^{d-2}_{q+1},z^{d-2}_{q+1},y^{d-3}_q\},\text{and}\\
		C^d&:=\{y^{d-1}_{2m},x^{d-2}_1,z^{d-2}_1\},
	\end{align*}
	for $1\leq p\leq m$, $1\leq q\leq 2m-1$, $i=1,2$ and $d\in 2\mathbb{Z}$. Note that any arrows between subsets must be diagonal. Disregarding the diagonal arrows between subsets, each complex of $A$'s and $B$'s has four generators arranged in a square, and each complex of $C$'s has three generators which looks like $St(1,1)$. Therefore, if we remove all arrows between subsets i.e. $CFK^\infty(D)$ is a direct sum of $A$'s, $B$'s and $C$'s, the theorem follows.
	
	Define a subset of $A$ as $A'^d_{p,i}:=\{v^{d}_{p,i+m},u^{d-1}_{p,i},w^{d-1}_{p,i}\}$. Due to grading constraints on the filtered complex, we observe the following:
	\begin{itemize}
		\item $\partial_D$ of any generator in $A'^d_{p,i}$ has components only in $A^d_{k,j}$, $B^d_{q}$ and $C^d$ for $k<p$, $j=1,2$, and $q=1,\dots,2m-1$ (i.e. diagonal arrows between $A$'s going from higher to smaller first index.)
		\item $\partial_D$ of the generators in $B$'s and $C$'s are zero. 
	\end{itemize}
	These observations allow us to apply Proposition \ref{cor2} inductively to remove all diagonal arrows in the complex.
	
	We start to remove any diagonal arrows from $A'^d_{m,1}$. First, we remove all diagonal arrows going from $A'^d_{m,1}$ to $A^d_{m-1,1}$, using basis-change of the case (I) of Proposition \ref{cor2}. Differently from the propostion, there can be other components in $\partial_D$ of elements in $A'^d_{m,1}$, not in $A^d_{m-1,1}$. However, considering the grading constraints again, one can easily check that the other components cannot induce a $z$ component of $\partial^2a$ in the Eq.\@ (\ref{eq:1}), hence the equation that $A+B+C+D=0$ in the proof of the corollary still holds and we remove diagonal arrows using a basis-change in the corollary.
	
	After applying the basis-change, two types of newer diagonal arrows will be added due to arrows coming to $A'^d_{m,1}$ and arrows going from $A^d_{m-1,1}$. First note that there are no diagonal arrows coming to $A^d_{m,1}$ by the observations above ($m$ is the greatest index for $A$.) Secondly, a diagonal arrow from $A^d_{m-1,1}$ to some generator adds an arrow going from $A^d_{m,1}$ to the generator after basis-change, but note that these arrows are going to the subsets $A^d_{p,i}$ with $p<m-1$ and $i=1,2$, which we will remove later.
	
	Now, we similarly change the basis for removing diagonal arrows from $A'^d_{m,1}$ to $A^d_{m-1,2}$, $A^d_{m-2,1}$, $A^d_{m-2,2}\dots,A^d_{1,1},$ and $A^d_{1,2}$ in sequence. Then, case (II) and (III) of the corollary will be applied to remove arrows from $A'^d_{m,1}$ to $B^d_q$'s and $C^d$. The induction ends with removing any $\partial_D$ from $A'^d_{m,1}$, since there are no diagonal arrows from $A^d_{1,i}$'s, $B^d_q$'s and $C^d$.
	
	Then, we remove $\partial_D$ of $A'^d_{m,2}$, $A'^d_{m-1d,1}$, $A'^d_{m-1,2},\dots,A'^d_{1,1},$ and $A'^d_{1,2}$ likewise. After removing the diagonal arrows from $A'_{p,i}$ for all $p=1,\dots,m$ and $i=1,2$, the only remaining non-trivial $\partial_D$ are ones of $v_{p,1}$ and $v_{p,2}$. It is easy to see that $\partial_D$'s of $v_{p,1}$ and $v_{p,2}$ also vanish: $0{=}\partial^2(u_{p,i}){=}\partial(v_{p,i})$. Thus, we may assume that $\partial_D$'s of $CFK^\infty$ are all zero.
\end{proof}
\begin{remark}
	The computation in this section is possibly generalized further to any {\it $L$-space knots} \cite{Ozsvath-Szabo:2005-1} i.e. one might prove that $CFK^\infty$ of the Whitehead double of a $L$-space knot is chain homotopy equivalent to  $CFK^\infty(T_{2,3})\oplus A$ for some acyclic complex $A$. 
\end{remark}

\section{$\delta$-invariant of $D^2(T_{2,2m+1})$ and proof of Theorem \ref{thm:torus}}
Now, we are giving the proof of Theorem \ref{thm:torus}. This will follow from an explicit computation of $\delta(D^2(T_{2,5}))$.

\begin{proposition}\label{prop:delta}
	$\delta(D^2(T_{2,2m+1}))=-4$ for $m\geq1$.
\end{proposition}
\begin{proof}
	By applying Lemma \ref{lemma:delta} to $D^2(T)$, we have 
	\[\delta(D^2(T))=2dS^3_1(D(T)\#D(T)^r),\]
	where $T=T_{2,2m+1}$. Observe that $d(S^3_1(K))$ is derived from a direct summand of $CKF^\infty(K)$ containing a generator of $H_*(CFK^\infty(K)_{\{i=0\}})\cong\mathbb{F}$. In particular, if $CFK^\infty(K)$ and $CFK^\infty(K')$ differ only by an acyclic complex, then $dS^3_1(K)=dS^3_1(K')$. 
	
	Now, let us understand $CFK^\infty(D(T)\#D(T)^r)$. Since the knot Floer complex is unchanged under the orientation reversal \cite[Proposition 3.9.]{Ozsvath-Szabo:2004-1} and by the connected sum formula for knot Floer complexes in \cite[Theorem 7.1.]{Ozsvath-Szabo:2004-1},
	\[CFK^\infty(D(T)\#D(T)^r)\cong CFK^\infty(D(T))\otimes CFK^\infty(D(T)).\]
	Thus, $dS^3_1(D(T)\#D(T)^r)$ equals $dS^3_1(T_{2,3}\#T_{2,3})$ by Theorem \ref{thm:complex}. It is computed that \[dS^3_1(T_{2,3}\#T_{2,3})=-2\] as an example of the computer program, {\texttt dCalc} in \cite{Peters:2010-1}, (it can be also computed by Proposition \ref{prop}) so that $\delta(D^2(T))=-4$. 
\end{proof}

\begin{proof}[Proof of Theorem \ref{thm:torus}] 
	The linear-independency of $D(T_{2,2m+1})$ and $D^2(T_{2,2m+1})$ in $\con$ for $m\geq3$ easily follows from Theorem \ref{thm:general}. Therefore, it suffices to show that $D(T_{2,5})$ and $D^2(T_{2,5})$ can generate $\bz\oplus\bz$-summand of $\conts$ to complete the proof. 
	
	Recall that $\delta\equiv\sigma/2$ mod 4 \cite[(2.1)]{Manolescu-Owens:2007-1} and $\sigma=0$ for any knot in $\conts$. Consider the homomorphism $\psi=(\tau,\delta/4):\conts\rightarrow\mathbb{Z}\oplus\mathbb{Z}$. Since $\psi(D(T_{2,5}))=(1,-2)$ and  $\psi(D^2(T_{2,5}))=(1,-1)$ by Theorem \ref{thm:Hedden}, Theorem \ref{thm:Manolescu_Owens} and Proposition \ref{prop:delta}, $\psi$ is surjective. Therefore, $\conts$ has a $\mathbb{Z}\oplus\mathbb{Z}$ summand generated by $D(T_{2,5})$ and $D^2(T_{2,5})$.
\end{proof}

\begin{proof}[Proof of Corollary \ref{cor:bipolar}]
	By \cite[Corollary 4.9, Corollary 6.11]{Cochran-Harvey-Horn:2013-1} both $\tau$ and $\delta$ vanish for the knots in $\mathcal{C}_{\Delta}\cap\cT_0$. Now, consider the induced homomorphism $(\tau,\delta/4)\colon\mathcal{C}_{\Delta}/(\mathcal{C}_{\Delta}\cap\cT_0)\rightarrow\bz\oplus\bz$, and the surjectivity of it can be shown again by the knots, $D(T_{2,5})$ and $D^2(T_{2,5})$.
\end{proof}

For the right-handed trefoil knot, as far as the author knows, all known concordance invariants of $D(T_{2,3})$ and $D^2(T_{2,3})$ are the same, so it is still mysterious if $D(T_{2,3})$ and $D^2(T_{2,3})$ are smoothly concordant.\\

\noindent{\bf Question.} Are $D(T_{2,3})$ and $D^2(T_{2,3})$ smoothly concordant? If not, are they linearly independent in $\con$?

\begin{remark}
	This question is possibly approached by using gauge-theoretic invariants. See for example Hedden and Kirk \cite{Hedden-Kirk:2012-1}.
\end{remark}

\bibliographystyle{alpha}
\bibliography{references}{}

\begin{thebibliography}{CHH13}

\bibitem[BL14]{Borodzik-Livingston:2014-1}
M.~Borodzik and C.~Livingston.
\newblock Heegaard {F}loer homology and rational cuspidal curves.
\newblock {\em Forum Math. Sigma}, 2:e28, 23, 2014.

\bibitem[CG86]{Casson-Gordon:1975}
A.~Casson and C.~Gordon.
\newblock Cobordism of classical knots.
\newblock In {\em \`A la recherche de la topologie perdue}, volume~62 of {\em
  Progr. Math.}, pages 181--199. Birkh\"auser Boston, Boston, MA, 1986.
\newblock With an appendix by P. M. Gilmer.

\bibitem[CHH13]{Cochran-Harvey-Horn:2013-1}
T.~D. Cochran, S.~Harvey, and P.~Horn.
\newblock Filtering smooth concordance classes of topologically slice knots.
\newblock {\em Geom. Topol.}, 17(4):2103--2162, 2013.

\bibitem[Fre82]{Freedman:1982-1}
M.~H. Freedman.
\newblock The topology of four-dimensional manifolds.
\newblock {\em J. Differential Geom.}, 17(3):357--453, 1982.

\bibitem[GS99]{Gompf-Stipsicz:1999}
R.~E. Gompf and A.~I. Stipsicz.
\newblock {\em {$4$}-manifolds and {K}irby calculus}, volume~20 of {\em
  Graduate Studies in Mathematics}.
\newblock American Mathematical Society, Providence, RI, 1999.

\bibitem[Hed07]{Hedden:2007-1}
M.~Hedden.
\newblock Knot {F}loer homology of {W}hitehead doubles.
\newblock {\em Geom. Topol.}, 11:2277--2338, 2007.

\bibitem[HK12]{Hedden-Kirk:2012-1}
M.~Hedden and P.~Kirk.
\newblock Instantons, concordance, and {W}hitehead doubling.
\newblock {\em J. Differential Geom.}, 91(2):281--319, 2012.

\bibitem[HKL16]{Hedden-Kim-Livingston:2016-1}
M.~Hedden, S.~Kim, and C.~Livingston.
\newblock Topologically slice knots of smooth concordance order two.
\newblock {\em J. Differential Geom.}, 102(3):353--393, 2016.

\bibitem[Hom15]{Hom:2015-1}
J.~Hom.
\newblock An infinite-rank summand of topologically slice knots.
\newblock {\em Geom. Topol.}, 19(2):1063--1110, 2015.

\bibitem[Hom17]{Hom:2017}
J.~Hom.
\newblock A survey on {H}eegaard {F}loer homology and concordance.
\newblock {\em J. Knot Theory Ramifications}, 26, 2017.

\bibitem[Kir97]{Kirby:1997}
R.~C. Kirby.
\newblock Problems in low-dimensional topology.
\newblock In {\em Geometric topology ({A}thens, {GA}, 1993)}, volume~2 of {\em
  AMS/IP Stud. Adv. Math.}, pages 35--473. Amer. Math. Soc., Providence, RI,
  1997.
\newblock Edited by Rob Kirby.

\bibitem[KP16]{Kim-Park:2016}
M.~H. Kim and K.~Park.
\newblock An infinite-rank summand of knots with trivial alexander polynomial.
\newblock {\em to appear in J. Symplectic Geom.}, 2016.
\newblock arXiv:1604.04037.

\bibitem[Lev69]{Levine:1969}
J.~Levine.
\newblock Invariants of knot cobordism.
\newblock {\em Invent. Math. 8 (1969), 98--110; addendum, ibid.}, 8:355, 1969.

\bibitem[Liv08]{Livingston:2008-1}
C.~Livingston.
\newblock Slice knots with distinct {O}zsv\'ath-{S}zab\'o and {R}asmussen
  invariants.
\newblock {\em Proc. Amer. Math. Soc.}, 136(1):347--349 (electronic), 2008.

\bibitem[LN06]{Livingston-Naik:2006-1}
Charles Livingston and Swatee Naik.
\newblock Ozsv\'ath-{S}zab\'o and {R}asmussen invariants of doubled knots.
\newblock {\em Algebr. Geom. Topol.}, 6:651--657, 2006.

\bibitem[MO07]{Manolescu-Owens:2007-1}
C.~Manolescu and B.~Owens.
\newblock A concordance invariant from the {F}loer homology of double branched
  covers.
\newblock {\em Int. Math. Res. Not. IMRN}, (20):Art. ID rnm077, 21, 2007.

\bibitem[OS03a]{Ozsvath-Szabo:2003-1}
P.~Ozsv{\'a}th and Z.~Szab{\'o}.
\newblock Absolutely graded {F}loer homologies and intersection forms for
  four-manifolds with boundary.
\newblock {\em Adv. Math.}, 173(2):179--261, 2003.

\bibitem[OS03b]{Ozsvath-Szabo:2003-2}
P.~Ozsv{\'a}th and Z.~Szab{\'o}.
\newblock Heegaard {F}loer homology and alternating knots.
\newblock {\em Geom. Topol.}, 7:225--254 (electronic), 2003.

\bibitem[OS03c]{Ozsvath-Szabo:2003-3}
P.~Ozsv{\'a}th and Z.~Szab{\'o}.
\newblock Knot {F}loer homology and the four-ball genus.
\newblock {\em Geom. Topol.}, 7:615--639, 2003.

\bibitem[OS04a]{Ozsvath-Szabo:2004-1}
P.~Ozsv{\'a}th and Z.~Szab{\'o}.
\newblock Holomorphic disks and knot invariants.
\newblock {\em Adv. Math.}, 186(1):58--116, 2004.

\bibitem[OS04b]{Ozsvath-Szabo:2004-3}
P.~Ozsv{\'a}th and Z.~Szab{\'o}.
\newblock Holomorphic disks and three-manifold invariants: properties and
  applications.
\newblock {\em Ann. of Math. (2)}, 159(3):1159--1245, 2004.

\bibitem[OS04c]{Ozsvath-Szabo:2004-2}
P.~Ozsv{\'a}th and Z.~Szab{\'o}.
\newblock Holomorphic disks and topological invariants for closed
  three-manifolds.
\newblock {\em Ann. of Math. (2)}, 159(3):1027--1158, 2004.

\bibitem[OS05]{Ozsvath-Szabo:2005-1}
P.~Ozsv{\'a}th and Z.~Szab{\'o}.
\newblock On knot {F}loer homology and lens space surgeries.
\newblock {\em Topology}, 44(6):1281--1300, 2005.

\bibitem[OS06]{Ozsvath-Szabo:2006-1}
P.~Ozsv{\'a}th and Z.~Szab{\'o}.
\newblock Holomorphic triangles and invariants for smooth four-manifolds.
\newblock {\em Adv. Math.}, 202(2):326--400, 2006.

\bibitem[OSS17]{Ozsvath-Stipsicz-Szabo:2014}
P.~S. Ozsv\'ath, A.~I. Stipsicz, and Z.~Szab\'o.
\newblock Concordance homomorphisms from knot {F}loer homology.
\newblock {\em Adv. Math.}, 315:366--426, 2017.

\bibitem[Pet10]{Peters:2010-1}
T.~D. Peters.
\newblock {\em Computations of {H}eegaard {F}loer homology: {T}orus bundles,
  {L}-spaces, and correction terms}.
\newblock ProQuest LLC, Ann Arbor, MI, 2010.
\newblock Thesis (Ph.D.)--Columbia University.

\bibitem[Ras03]{Rasmussen:2003-1}
J.~A. Rasmussen.
\newblock Floer homology and knot complements.
\newblock page 126, 2003.
\newblock Thesis (Ph.D.)--Harvard University.

\bibitem[Ras10]{Rasmussen:2010-1}
J.~A. Rasmussen.
\newblock Khovanov homology and the slice genus.
\newblock {\em Invent. Math.}, 182(2):419--447, 2010.

\end{thebibliography}

\end{document}